\documentclass[11pt]{article}
\usepackage{amsmath,amssymb,amsfonts,amsthm}
\usepackage{graphics}
\usepackage{epsfig}
\usepackage{amsfonts}

\font\elevensf=cmss10 scaled\magstephalf

\newtheorem{theorem} {{\elevensf THEOREM}}[section]
\newtheorem{proposition} {{\elevensf PROPOSITION}}[section]

\newtheorem{lemma} {{\elevensf LEMMA}}[section]

\newtheorem{remark} {{\elevensf REMARK}}[section]

\newtheorem{definition} {{\elevensf DEFINITION}}[section]

\def\R{{I\!\!R}}

\def\N{{I\!\!N}}
\def\CC{{\rm \kern.24em \vrule width.02em height1.4ex depth-.05ex \kern-.26emC}}

\def\TagOnRight

\def\AA{{it I} \hskip-3pt{\tt A}}

\def\QQ{\rlap {\raise 0.4ex \hbox{$\scriptscriptstyle |$}} {\hskip -0.1em Q}}

\newcommand{\be} {\begin{equation}}
\newcommand{\ee} {\end{equation}}
\newcommand{\bea} {\begin{eqnarray}}
\newcommand{\eea} {\end{eqnarray}}
\newcommand{\Bea} {\begin{eqnarray*}}
\newcommand{\Eea} {\end{eqnarray*}}
\newcommand{\pa} {\partial}

\newcommand{\al} {\alpha}

\newcommand{\de} {\delta}

\newcommand{\Ga} {\Gamma}
\newcommand{\Om} {\Omega}

\newcommand{\De} {\Delta}
\newcommand{\la} {\lambda}

\newcommand{\ta} {\theta}

\newcommand{\noi} {\noindent}

\newcommand{\na} {\nabla}

\newcommand{\var} {\varepsilon}
\catcode`\@=11

\def\theequation{\@arabic{\c@section}.\@arabic{\c@equation}}

\catcode`\@=12
\begin{document}
\baselineskip12pt
\parskip10pt
\parindent.4in
\catcode`\@=11

\begin{center}
{\large \bf  A  Reduction Method for Semilinear Elliptic Equations and Solutions Concentrating  on  Spheres}\footnote{Work supported by PRIN-2009-WRJ3W7 grant.} \\[3mm]
by \\[3mm]
Filomena  Pacella \\

Dipartimento di Matematica, Universit$\bar{\rm a}$ di Roma \lq\lq Sapienza"\\
P.le A. Moro 2-00185, Roma -Italy\\
E-mail: pacella@mat.uniroma1.it\\[3mm]
and \\[3mm] P.N. Srikanth\\

TIFR-CAM, Sharadanagar,  Chikkabommasandra, Bangalore -560 065 \\
E-mail:srikanth@math.tifrbng.res.in \end{center}
\begin{abstract}
 We show that any general  semilinear  elliptic problem with  Dirichlet or Neumann boundary conditions  in an annulus
$A \subseteq \R^{2m}, m\geq 2,$ invariant by the action of a certain symmetry group can be reduced to a
nonhomogenous similar problem in an annulus $D \subset \R^{m+1},$ invariant by another related  symmetry. We apply
this result to prove the existence of positive and sign  changing solutions of a singularly perturbed elliptic problem in $A$ which
concentrate on one or two $(m-1)$ dimensional spheres. We also prove that the Morse indices of these solutions  tend to
infinity as the parameter of concentration tends to infinity. \end{abstract}

{\bf 2010 AMS Classification:} 35J61, 35B25, 35B40


\section{Introduction}

In this paper we propose a method to reduce a  semilinear elliptic problem of the type~:
\be \left\{  \begin{array}{lllll}
 -\Delta u  =  f (u)  &\mbox{ in }&A \subseteq \R^{2m} \\[2mm]
u = 0  \  \  or \ \ \displaystyle{  \frac{\pa u}{\pa  \nu}} =0  &\mbox{ on } \pa A\end{array} \right.\ee
where $A$ is an annulus in $\R^{2m}, m \geq 2, $
$$ A =\{ x \in \R^{2m}:  \quad  a < |x| < b,   \ \ \  0 < a < b\}  $$ and  $f $ is a $C^{1, \al}$   nonlinearity, to the
semilinear elliptic problem:
\be
\left\{ \begin{array}{llll} -\De  v= \displaystyle{ \frac{f (v)}{2 |z|}}  &\mbox{ in } & D\subseteq \R^{m+1}  \\[2mm]
v=0  \mbox{ or } \displaystyle{\frac{\pa v}{\pa  \nu}} =0 &\mbox{ on } \pa D \end{array} \right. \ee
 where $z \in  \R^{m+1} $ and $D$ is the annulus
$$D =\left\{ z \in \R^{m+1}  \ : \ \frac{ a^2}{2} < |z| < \frac{b^2}{2}  \right\} $$
As  will be clear from the construction, there will be a one to one correspondence between solutions of (1.1)   invariant under
 the  action of a symmetry group in $H_0^1 (A)$ or  $H^1 (A)$  and solutions of (1.2) invariant under another symmetry in $H^1_{0} (D) $ or
 $H^1 (D).$

\par More precisely, writing $x \in \R^{2 m} $ as  $x = (y_1, y_2), y_i \in \R^m, i =1, 2.$  we consider
solutions $u$ of (1.1) which are radially symmetric in $y_1$ and $y_2$  i.e  $u(x) = w (|y_1|,  |y_2|)$ and
solutions  $v$   of   (1.2) which are  axially  symmetric i.e. $v (z) = h (|z|, \varphi)$ with   $\varphi  = \arccos  \left(
\frac{z}{|z|} \cdot p\right) $  for a unit vector $p \in  \R^{m+1}. $

\par Since  the domains   are annuli, by standard regularity  theory all solutions we consider are classical  $C^{2,\alpha}$-
solutions. We  set:
$$ \begin{array}{lllll}
 X = \left\{ u \in  C^{2, \alpha}  (\overline{A})\ : \  u  (x) = w (|y_1|, |y_2|) \right\}  \\ [2mm]
Y= \displaystyle{\left\{ v \in C^{2, \alpha} (\overline{D}) \ : \ v  \mbox{ axially  symmetric } \right\}}  \end{array}$$

 Our result is the following :

\begin{theorem}
 There is a  bijective correspondence  between solutions of (1.1) in $X$ and  solutions of
(1.2) in $Y.$
\end{theorem}

The map   which gives the bijection to prove Theorem 1.1 will be    defined in Section 3 after choosing  suitable
coordinates  in $\R^{2m} $ and $\R^{m+1}.$

\par The  possibility  of reducing a problem  in dimension $2m$ to a problem  in the  lower dimension $(m+1)$ is of
 great importance   in the study of semilinear elliptic equations. As example   one can   think of the case of power
nonlinearities  when  critical  or supercritical   problems in $\R^{2m}$  can  become subcritical in $\R^{m+1}.$ Moreover
solutions   concentrating  on sets  of   a certain  dimension in $\R^{m+1}$ (e.g. points) can give rise to  solutions
concentrating  on higher dimensional  manifolds on $\R^{2m}.$

Indeed, the inspiration for our method came from the paper  \cite{RS} where a reduction method was introduced to
pass from a singularly   perturbed problem  in an annnulus in $\R^4$ to a singularly perturbed problem in an annulus in
$\R^3$ which allowed  to prove the existence of  solutions  concentrating  on $S^1$-orbits in $\R^4.$ Their
construction is  related to Hopf  fibrations   and  an  extension  to other dimensions seems to be possible only in dimension
8  and 16 (see also \cite{CFP}).

Our reduction  works  in all even dimensions, and  in $\R^4$ allows  to get the  same result as in \cite{RS}.

The new  idea (compared to \cite{RS}) is  to impose more symmetry  on both problems  and the   key point is to
identify  symmetric solutions of (1.1) with axially symmetric  solutions of (1.2), once the   reduction has been made. In \cite{RS} the  reduced problem  in $\R^3$ did  not have any particular symmetry.

Of course, this means that   our method can be applied to \lq\lq lift" solutions from $D \subset  \R^{m+1}$ to     $A \subset
\R^{2 m} $ when we know that the solutions in $D$ are axially symmetric. However, by the results of \cite{P}  and
\cite{PW}  (see also \cite{GPW}) this  is true   for every solution of (1.2) with  Morse index less than on equal to $(m+1)$ if $ f$ or
$f^\prime_{v}$ are  convex. Actually in these papers only the Dirichlet problem is  considered, but it is easy to see,
arguing as in \cite{M}, that the symmetry results extend also to the  Neumann problem.      In particular the solutions in the
annulus in $\R^3$ considered in \cite{RS} which  are  least energy positive solutions are axially symmetric. This is why our result applies to the case of  \cite{RS}.
As application of our reduction method we focus as in \cite{RS}, on the following  Dirichlet  problem which was  , indeed, the
initial motivation  for our  result:
\be
\left\{ \begin{array}{llllll}
-\Delta u  +\la u = |u|^{p-1} u  &\mbox{ in A}  \\
u =0 &\mbox{  on }  \pa A \end{array}\right.  \ee

where $\la >0$ and  $p>1.$

Note that (1.3) is equivalent  to the singularly perturbed problem

\be \left\{ \begin{array}{rlllll}
 -\varepsilon ^2 \De u + u = |u|^{p-1} u  &\mbox{  in } & A \\
 u &=& 0 &\mbox{ on }  \pa  A  \end{array}\right.  \ee
and to study (1.3)  as  $\la \rightarrow \infty$ is equivalent to studying (1.4) as $\varepsilon \rightarrow  0.$
By applying  Theorem 1.1. we  reduce problem (1.3) to (1.2) with
\be
f (v) = |v|^{p-1}  v -\lambda v \ee

If we  take $p < \frac{(m+1)+2}{(m+1)-2},$ i.e  $p$ subcritical  in dimension $(m+1)$,   then the  inhomogeneous  problem  (1.2),  with  $f$ given by  (1.5), can be studied  as in \cite{RS}
adapting the   methods of \cite{NW}, \cite{DF} to produce least energy single peak positive solutions $v_\la$ which
concentrated on the  inner boundary of $D$ as $ \la \rightarrow  \infty.$ Also,  by adapting the method of
\cite{Nou-W} we show the existence  of least energy two peaks nodal solutions  $\tilde{v}_\lambda$ of (1.2) which
again concentrate on the inner boundary as $\la  \rightarrow  +\infty.$

Since these solutions have Morse index 1 and 2 respectively, by the results of \cite{P} and \cite{PW}  (see also \cite{BWW}) we
deduce that  they are foliated Schwarz symmetric, so, in particular , they are axially  symmetric.

Then  we  can   transform  these solutions getting families of  positive solutions of (1.3) concentrating on a
 $(m-1)$  dimensional - sphere and families of   sign changing  solutions  of (1.3)    concentrating on two $(m-1)$-  dimensional spheres. So,
denoting by $(\pa A)_a =\{x  \in \pa A, |x| = a \}, $ we have

\begin{theorem}
 Let $1 <p < \frac{(m+1) +2}{(m+1) -2} $. Then there exists a family $ \ \{u_\la\}$ of
positive solutions of (1.3) concentrating on  a $(m-1) -$ dimensional sphere $\Ga \subset (\pa A)_a$ and a family $
\{\tilde{u}_\lambda\}$ of sign changing solutions of  (1.3) such that the positive   part   $ \{\tilde{u}^+_\la\}$   and  the negative part $\{\tilde{u}_{\la}^-\} $  concentrate on
$(m-1)$ dimensional spheres $\Ga^+ \subset (\pa A)_a$ and   $\Ga^-  \subset (\pa A)_a,$
respectively,  as $\la  \rightarrow   +\infty. $
\end{theorem}

Let us remark that the exponent $p \in\left( 1,  \frac{ (m+1) +2}{ (m+1) -2} \right)$ can be critical or supercritical for  problem
(1.3)  in $\R^{2m.}$

\begin{remark}

 As it  is clear from the construction these solutions   have an $O (m) \times O (m)$ symmetry.
\end{remark}
As far as  we know the result  of Theorem 1.2  is  the first result   for  singularly perturbed Dirichlet problems about  sign changing  solutions concentrating on  manifolds   of dimension larger    than or equal to 1. It is also a new result for positive solutions
since the only previous    ones  concern  concentration on  $(2m - 1)$ dimensional spheres (\cite{AMN}, \cite{EMSS}) or on 1-dimensional spheres  in $\R^4 (\cite{RS}). $

\par Another interesting question connected with the concentration phenomena is the  asymptotic behaviour of the Morse
index of the solutions when $\la \rightarrow  +\infty.$
For  the least energy positive or sign changing solutions concentrating in one or two points the Morse index is obviously
independent of $\lambda$ and is $1$ or $2$.  When the concentration takes place on a $(2m -1)$-dimensions sphere, which is the case of radial solutions,    then it is easy to see that the  Morse index tend to  infinity, as $\lambda \rightarrow +\infty.$ Indeed in this  case the
spectrum of the  linearized  operator can be split into a \lq\lq radial" part and an \lq\lq angular" part.  For  solutions concentrating  on lower dimensional  spheres, as
in our case, a decomposition of the spectrum does not seem immediate.
However we are able  to show:

\begin{theorem}
  The Morse indices  $m (u_\lambda) $ and $ m  (\tilde{u}_\la)$ of the solutions constructed  in
Theorems  1.2 tend to infinity as $\la \rightarrow +\infty.$
\end{theorem}

To prove  Theorem 1.3. we test the quadratic form associated to the  linearized operator at $u_\la$  and
$\tilde{u}_\la$ by some functions  obtained   using eigenfunctions  of the  Laplace -Beltrami operator on $S^{m-1}.$ Moreover we exploit that the first
eigenvalue of the linearised operator at these solutions   tends to $-\infty$ as $\la \rightarrow  \infty.$
\par  The  result of Theorem 1.3 shows that the
concentrating solutions we find become more and more unstable as $\la \rightarrow   \infty. $ This indicates that many
local bifurcations  should  occur. However,  since we do not know if that the families $\{u_\la \} $ or $
\{\tilde{u}_\la \}$ give a   curve  in $X,$  we cannot prove it rigorously. Note that if we knew  that the least energy
solutions (positive  or   nodal) in $D$  were unique (up to symmetry) then this  would be  true.

\par Finally  we believe that this reduction method can be used to get different kind of results for other  semilinear elliptic problems (as for example  in \cite{CFP}). Moreover it should be possible to  generalize   this approach      to reduce problems in $\R^k$ to
problem in $\R^h,$ for suitable   $h <k$ exploiting other  symmetries.

\par The paper is  organized   as follows. In Section  2 we recall some  results   on symmetry of solutions of general semilinear   elliptic
equations . In Section 3  we introduce suitable coordinates and symmetries and prove Theorem 1.1.   In Section 4 we show the concentration   of the least energy solutions  of problem  (1.2).  Finally in Section 5 we prove Theorem 1.2 and Theorem 1.3.

\section{ Axial Symmetry  of solutions of semilinear elliptic  equations. }
\setcounter{equation}{0} Let  us consider a general semilinear   elliptic problem   of the type:
 \be \left\{\begin{array}{lllll}
-\De v = f (|z|,  v)  &\mbox{  in B}  \\
v =0   \mbox{ or }  \displaystyle{\frac{\pa v}{\pa  \nu}} =0  &\mbox{ on } \partial  B   \end{array} \right. \ee
where $B$ is either an annulus or a ball centered at the  origin of $\R^N, N \geq  2, z \in \R^N, $ and  $f:
\overline{B} \times \R \rightarrow \R$ is (locally ) a $C^{1,\alpha}$-function  .
We give the following  definitions:
\begin{definition}
 We say that a function $ v \in C (\overline{B})$ is  axially  symmetric if there is a unit vector
$p \in \R^N, |p| =1$ such that  $v (x) $ only depends on $\rho =|z| $ and      $\varphi  =\arccos \left(\frac{z}{|z|}. p\right).$
\end{definition}

\begin{definition}
If an axially symmetric  function is also non increasing in the polar angle     then it is called  foliated  Schwarz symmetric.
\end{definition}

\begin{remark}
Let us write  $z \in \R^N $ as  $z= (z_1, \ldots z_n)$  and consider the  spherical coordinates
$(\rho, \varphi_1, \varphi_2, \ldots \varphi_{N-1}),  \varphi_i \in [0, \pi], i = 1, \ldots N-2, \ \varphi _{N-1} \in
[0, 2\pi], \rho = |z|,$ then
\be \left\{\begin{array}{llll}
z_1  = \rho \sin \varphi_1 \ldots \sin \varphi_{N-1} \\
 z_2 = \rho \sin \varphi_1 \ldots  \sin \varphi_{N-2} \cos \varphi_{N-1} \\
 \ldots \ldots \ldots \\
z_{N-1} =\rho \sin  \varphi_1 \cos \varphi_2 \\
z_N =\rho \cos \varphi_1  \end{array} \right. \ee

Then if $v$  is a  axially symmetric function, without loss of generality, we  can think  that the  vector $p$  is
$ p= (0, 0, \ldots 1)$ i.e.  the symmetry axis is the $z_n$-axis  and therefore  $v$ depends  only on $\rho $ and
$\varphi_1.$
\end{remark}

\begin{remark}

 If $v$ is an axially symmetric function   belonging   to $C^{2,\al}   (\overline{B})$ then the  Laplace operator, using
the above coordinates and the fact that $v =v (\rho, \varphi_1),$ reduces to

\be
\Delta_{R^N} v = v _{\rho \rho} + \frac{N-1}{\rho} v_\rho +
\frac{N-2}{\rho^2} \frac{\cos \varphi_1}{\sin \varphi_1} v _{\varphi_1}  +
\frac{1}{\rho^2} v_{\varphi_1 \varphi_1}  \ee
\end{remark}
Some  sufficient  conditions on the nonlinearity   and on a solution $v$ of  (2.1) for the  foliated Schwarz symmetry  have
been obtained in \cite{P}, \cite{PW} (see also \cite{GPW})  for  the
Dirichlet problem and they extend easily to the Neumann problem (see
\cite{M}).

We recall them here:

\begin{theorem}
Let $f (|z|,s)$ be either convex in the $s$-variable or with a  convex first derivative  $f^\prime (|z| , s)  =
\frac{ \pa    f}{\pa s} (|z|, s),  $ for  every  $z \in B$. Then any classical solution of (2.1) with Morse index $j \leq N $ is
foliated  Schwarz symmetric.
\end{theorem}

We recall that the Morse  index of a solution $v$ of (2.1) is the  number of the negative  eigenvalues of the
linearized operator at the solution $L_v =-\De - f^\prime (|z|, v)$ with  the same  boundary condition on $\pa  B.$

\par For least  energy  nodal solutions of (2.1) it is also useful to recall a similar result obtained  in \cite{BWW}.

\begin{theorem}
If  $f (|z|, s)$ is subcritical i.e $\exists  p \in \left(2, \frac{2N}{N-2}\right),$  if $N \geq 3, p \in  (2, \infty)$, if $N \leq 2 $ and $C >0$ such that
$$f  (|z|, s ) \leq  C (|s| + |s|^{p-1}) $$
and   if the function $s \rightarrow  \frac{f (|z|, s)}{ |s|}  $ is  strictly increasing     on $\R^{-}$ and $ \R^+  \ \
 \mbox{ for all }  z \varepsilon B, $ then the least  energy nodal solution of (2.1) is   foliated Schwarz symmetric.
\end{theorem}

\section{Reduction and proof of Theorem 1.1.}
\setcounter{equation}{0}
 Let us consider  $\R^{2m},  m \geq 2, $ as the  product of two
copies  of $\R^m$, i.e  $ \R^{2m}= \R^m \times \R^m$ and denote a point $x \in \R^{2m} $ by $x = (y_1, y_2), y_i \in
\R^m,   i =1,2.$

\par   Taking in  each $\R^m $ the  spherical coordinates
$$(\rho_1, \theta^1_1, \ldots, \theta^1_{m-1}), (\rho_2, \theta^2_1, \ldots, \theta^2_{m-1})$$
$$ \rho_1 = |y_1|, \quad \rho_2 = |y_2|, \quad \theta^i_1 \in [0, 2 \pi],  $$
$$\theta^i_j \in [0, \pi]   \mbox{ for  }  i =1, 2,  \quad  j=2, \ldots, m-1. $$
(see  Remark 2.1) and  observing  that
$$
\rho_1 =r \cos  \theta, \quad \rho_2 =  r   \sin \theta,  \  r  = |x|,  \  \theta  \in
[0, \frac{\pi}{2} ] $$
we have   that  a point $x  \in \R^{2m}$ can be  represented by the coordinates
\be
 x=  (r, \ta^1_1,    \ldots \ta^1_{m-1}, \ta^2_1, \ldots \ta^2_{m-1},   \theta) \ee
Then   if we  consider the annulus $A \subset  \R^{2m},  $
$$A =\{x \in \R^{2m},   a <  |x| < b\}, \quad  0  < a < b\} $$
a   function $u  \in  C^{2, \alpha} (\overline{A})$ which is invariant  under rotations in $y_1$ and
$y_2, $ ie.
\be
 u \in X = \left\{u \in C^{2, \alpha } (\overline{A}): u (x = w (|y_1, |y_2|) \right\} \ee
in the above  coordinates  will   depend only on   $r$ and  $ \theta,$ ie. $u =u (r, \theta)$

Therefore  for such  functions  the Laplace operator in $\R^{2m},$   written  in the above coordinates, reduces to

$$  \De_{\R^{2 m}} u \  \ = \  \ u_{r r} + \frac{(2m-1)}{r} u_r  + \frac{(m-1)}{r^2}  u_\theta \left[ \frac{\cos \theta}{\sin \theta}  - \frac{\sin \theta}{\cos \theta } \right]  + \frac{u_{\theta \theta}}{r^2} $$

Now we are ready  to prove Theorem 1.1.\\

\noi {\bf  Proof of Theorem 1.1}
Let  $u$ be a solution of (1.1) in $X$ and define the new variables:
\be
\rho =\frac{1}{2} r^2, \quad \varphi = 2 \theta  \ee
and the function
\be
 v (\rho, \varphi ) = u  (r (\rho), \theta  (\varphi )) =u  (\sqrt{2 \rho }, \frac{\varphi}{2 }) \ee
By  easy   computations we have
$$u_r = v_\rho \sqrt{2 \rho}, \quad  u_{rr} = 2  \rho v_{\rho \rho} + v _\rho $$
$$u_{\theta } = 2 v_\varphi , \quad  u_{\theta\theta} = 4  v _{\varphi \varphi} $$
Therefore , by (3.2) and (1.1) we get  that $v$ satisfies

$$ - 2 \rho \left[v_{\rho \rho} + \frac{m v_\rho}{\rho} + \frac{(m-1)}{\rho^2}  v_\varphi
\frac{\cos \varphi}{\sin \varphi} + \frac{v _{\varphi \varphi}}{\rho^2} \right]  = f (v), \quad
\rho \in \left( \frac{a^2}{2},  \frac{b^2}{2} \right) , \varphi \in [0, \pi] $$
Thus, by Remark 2.2 (with $N=m+1, \varphi_1 = \varphi, B=D) $ the function $v (\rho, \varphi)$  is an
axially  symmetric  solution  of  (1.2) in $D \subset \R^{m+1},$  i.e belongs to $Y$.  On the other hand starting, with an
axially symmetric solution $v = v (\rho, \varphi) $ of  (1.2) and defining $u(r, \theta)  = v (\rho (r), \theta (\varphi))$
with the same change of  variables   we get that $u \in X$ and is a solution of (1.1) invariant  by rotation in $y_1$ and $y_2$.
Hence the theorem holds.
\hfill{\qed}
\begin{remark}
 To understand better the   transformation (3.3)  let us   consider in $\R^{2m} $ the group action:
$$ T(x) = T (y_1, y_2) = (T_1 (y_1), T_2 (y_2)) , \quad   T_i  \in  O (m) , \ i=2. $$ It is  easy to see that  this action does not
have fixed  points in the annulus A.  Moreover the orbit  of a point $Q = (q_1, q_2) \in A, q_1 \in \R^m, q_2 \in \R^m$ is either $S^{m-1} \times  S^{m-1}$ if
$|q_i| \neq  0, i =1, 2,$ or just  $S^{m-1}$ if one  among $|q_1|$ or $|q_2|$ is zero.

\par Analogously we could  consider in $\R^{m+1} =\R^m \times  \R$  the group  action:
$$\tau (z) =\tau (z_1, \ldots, z_m, z_{m+1}) =(\tau_1 (z_1,\ldots z_m), z_{m+1}), \quad  \tau_1 \in (0 (m)).$$ In this case  all points  of the $z_{m+1}$ axis are fixed by this action.
Therefore  by the change  of variables (3.3)  used in the proof of Theorem 1.1 it is easy  to understand  that any   point
$P$ on the $z_{m+1} $ axis in $D \subset \R^{m+1},$ i.e any fixed point under   rotations   about the $z_{m+1} $ axis in the
annulus $D \subset \R^{m+1}$ is  mapped into  an $S^{m-1}$ orbit in $A \subset \R^{2m}.$ Indeed  $P$ has
spherical coordinates  in $\R^{m+1} $ equal to
$$P \equiv (\rho, 0, 0, \varphi), \mbox{  with }  \rho = |z_{m+1}|, \quad \varphi =0  \mbox{ or } \varphi =\pi.$$
and hence corresponds  in $\R^{2m} $ to points $Q$ with $\theta $ coordinate equal  to $0$ or $\frac{\pi}{2}.$  Therefore,
taking   in each $\R^m$  the spherical coordinates  (see the beginning of this section)
$$
(\rho_1, \theta^1_1, \ldots , \theta^1_{m-1} ),\quad  (\rho_2, \theta^2_1, \ldots, \theta^2_{m-1} ) $$
either  $\rho_1 $ or  $\rho_2$ (but never both!)  will be zero.
\end{remark}

%%there is also a bijection betwen the
%eigenfunctions in $X$ of he linearised opertor $L_u =-\De - f  (u)$  and the  %eigenfunctions in  $Y$ of the linearized
%%begin{corollary}
%A   function $\Phi$ in $X$ is an eigenfunction   of $L_u$  if and only if  the %function $\psi (\rho,\varphi) =
%\Phi(r (\rho), v (\varphi) =\Phi (\sqrt{2 \rho}, \varphi/2) $ is  an %eigenfunction of $L_v $ in  $Y  $  (see (3.3), (3.4)).
%Moreover the eigenvalues are the same.
%\end{corollary}

%\begin{proof}
%Arguing  as in the  proof of Theorem 1.1 we have  that  if $\varphi  \in X$ and
%$$
%L_u (\varphi)  =-\De \varphi  - f^\prime (u) \varphi = \mu \varphi \mbox{ for %some  }  \mu \in
%\R$$ then $\psi  \in Y$ and satisfiess
%$$L_v (\psi  ) = -\De \psi -\frac{1}{2 |z|} f^\prime  = \mu  \psi $$
%starting  instead from $\psi$ and using  the inverse  transformation we get the %reverse implementation.
%\end{proof}

\section{Concentrating Solutions}
\setcounter{equation}{0}

In this section we consider  the problem
\be \left\{\begin{array}{lllll}
-\De v  = \displaystyle{  \frac{1}{2|z|}  \left[ |v|^{p-1}  v -\la  v\right] \mbox{ in } D }\\
v =0 \mbox{ on  } \pa  D \end{array} \right. \ee
where $D =\{ z \in \R^N, R_1 < |z| <R_2\}  \quad0 < R^1 < R^2,  \quad
1 <p <\frac{N+2}{N-2}, \quad N \geq 3, \la >0. $
By known results we have
\begin{proposition}
For     every $ \la  >0$ problem  (4.1) has  a positive  solution  $v_\la$ and a sign changing   solution
$\tilde{v}_\la$ such that \\
i) $v_\la$  minimize the functional
\be
 J_\la  (u) =\int\limits_{D} \left[ \frac{1}{2} |\nabla v|^2 + \frac{\la  }{ 4   |z|} v^2 - \frac{ 1}{2 (\rho+1) |z|}
|v|^{p+1} \right]  \ee

 on the   Nehari manifold, in $H_0^1 (D),$
\be
N_\la =\{v \in  H_0^1  (D): v \neq 0,   \langle  J _\la^\prime (v), v  \rangle =0 \}  \ee

\noi (ii)  $\tilde{v}$ minimize $J_\la (v) $ in $H_0^1 (D) $ on the  nodal  Nehari set
\be
 N_\la^+ =\{ v \in H_0^1 (D): v^\pm \neq 0,  \langle J^\prime_\la,  v^\pm ,  v^\pm \rangle   =0 \}
\ee
where $v^\pm $ denotes either the positive or the negative part  of   $v$. \\

(iii)  $v_\la$ has Morse index 1 while  $\tilde{v}_\la$  has Morse index 2 and only two nodal regions.

(iv) $v_\la $ and $\tilde{v}_\la$ are  foliated  Schwarz symmetric.
\end{proposition}

\begin{proof}
Since  $p < \frac{ N+2}{N-2} $ (i) is a standard result in critical point theory.  The existence of
$\tilde{v}_\la$ satisfying (ii) is proved in  \cite{CCN} and \cite{BWW}. The  Morse index claim  (iii) is again classical for
$v_\la$ and   proved in \cite{BWW} for $\tilde{v}_\la$  where it is also proved that $\tilde{v}_\la$ has only  two  nodal regions.  Finally the   foliated  Schwarz symmetry  of  $v_\la$  and
 $\tilde{v}_{\la} $ is a consequence  of (iii) and Theorem 2.1 and Theorem 2.2.
\end{proof}

We  are interested in the  asymptotic behaviour of $v_\la$ and   $\tilde{v}_\la$
  as $\la \rightarrow  +\infty.$
For the positive  solution $v_\la$  we have

\begin{theorem}
 For $\la$ sufficientlly large   :\\
(i) $v_\la$  has  only one local maximum point $P_\la \in D$ and $\sqrt{\la}  d  (P_\la, \pa D)
\rightarrow  +\infty $ as  $\la \rightarrow  \infty$, where   $d (\cdot, \rho, \pa D)$ denotes  the distance from $\pa D$. \\
(ii)  $P_\la$   belongs to the symmetry axis  of $v_\la, $
$$   P_\la     \rightarrow   P \in \{z \in  \pa D,  |z| = R_1\} \mbox{ and }  v_\la
\rightarrow   0 \mbox{ in }  C^1_{loc} (D \setminus  \{P\}),   \mbox{ as } \la \rightarrow  +\infty.$$
\end{theorem}

\begin{proof}
 The result (i) is proved in \cite{RS}, adapting a theorem of \cite{NW}.  The  location of $P_\la$ on the symmetry axis is a consequence  of the
foliated  Schwarz  symmetry  since it implies  that all critical points  of $v_\la$ are on the  symmetry  axis.  Finally the
convergence of $P_\la$ to a  point  on the inner boundary  and the concentration of $v_\la$ in $P$ have  been
 proved  in \cite{RS} again following  the proof of \cite{NW}  and \cite{DF} .
\end{proof}

To  study the  asymptotic   behaviour of the least energy nodal solution $\tilde{v}_\la$ of (4.1),  as $\la \rightarrow \infty$  we
adapt the proofs of    \cite{Nou-W} where the asymptotic behaviour  of the least  energy nodal solution is studied  for the
autonomous  singularly   perturbed   problem

\be \left\{ \begin{array}{llll}
 -\var^2 \De v  + v  = |v|^{p-1}  v  &\mbox{  in } &  \Om \\
 v   = 0  &\mbox{ on } \pa  \Om.  \end{array} \right. \ee
 in a smooth bounded  domain $\Om. $ We also  use the modifications of the proofs of \cite{NW} and \cite{DF}  made  in \cite{RS}.

\par To start with, let us  assume, without loss of generality, that the symmetry axis of $\tilde{v}_\la$ is the
$x_N$-axis. Then, as a consequence of the  foliated  Schwarz symmetry of  $\tilde{v}_\la$ we have that all critical
points belong to the $x_N$-axis and,  by  the monotonicity with respect   to the polar angle (see  Definition 2.2) all local
maximum points $P^+_\la$ are on the set $\{x  = (x_{1}, \ldots, x_N) \in D, x_N > 0\}$ while all local  minimum points   belong to the
set  $\{ x = (x_{1}, \ldots, x_N) \in D,   x_N  <0 \} $. This implies that   $|P^+_\la -P^-_\la | > 2 R_1$ for all such points so that
they cannot  converge to the same point as $\la \rightarrow  +\infty.$
We can prove
\begin{theorem}
(i) For $\la$ sufficiently large $\widetilde{v}_\la$ has only one positive local maximum  point  $P^+_\la$ and only  one
negative local minimum point $P^-_\la. $ Moreover $\sqrt{\la} d (P^\pm_\la , \pa D) \rightarrow \infty $ as
$\la
\rightarrow  \infty, $ as before  $d (., \pa D) $  denotes the distance from  $\pa D.$ \\

(ii)  Let $P^+$ and $P^-$ be the  limit points of $P^+_\la$ and $P^-_\la$ respectively as $\la \rightarrow   \infty.$ Then
$\tilde{v}_\la  \rightarrow  0$ in $C^1_{loc} (D \setminus \{P^+, P^-\})  $ \\

(iii) $P^+$ and $P^-$  belong  to set   $\{z \in \pa D, |z| =R_1\}$
\end{theorem}
\begin{proof}
 The  first part of  assertion $i$) is similar  to the proof   of Lemma  3.3 in  \cite{Nou-W}  which in turn, uses the same
argument of  \cite{NW} . The fact that $\sqrt{\la}  d (P^\pm_\la, \pa  D) \rightarrow  +\infty$ as $\la
\rightarrow +\infty$ can be deduced  as in \cite{RS} (Proposition 4 there) by using   a \lq\lq boundary straightening",  a
 rescaling  argument and the fact that limit problem
$$ \left\{ \begin{array}{lllll}
 -\De u + \frac{u}{2 |P^\pm|}  - \frac{u^P}{2 |P^{\pm}|} = 0  \mbox{ in } \Om \\
u > 0 \\
u =0  \mbox { on } \pa \Om \end{array} \right. $$
 does not  have solutions   if $\Om$ is the
half-space, by Theorem  1.1 of \cite{EL}.
\end{proof}

To prove ii)  we use the modification of Proposition 3.4   of  \cite{NW}  derived in \cite{RS} and apply it to the positive part
$\tilde{v}^+_\la$ and to the negative part $\tilde{v}_{\la}^-.$ This is based on the  comparison with the
known radial   solution $w_d$ of the problem.
\be \left\{ \begin{array}{llll}
-\De  w + \frac{1}{2d} w -\frac{1}{2d} w^P =0  \mbox{ in } \R^N \\[2mm]

w (x) \rightarrow  0 \mbox{ as } |x| \rightarrow  \infty \end{array} \right. \ee
where $d$ is either $|P^+|$ or $|P^-|$.
Thus, defining  $\tilde{w}^\pm_\la  (y) = \tilde{v}_\la  \left( P^{\pm}_\la  +
\frac{y}{\sqrt{\la} } \right)  $ we get, as in Proposition 3.4 of \cite{RS}: \\

 For any $\de \in (0, 1) \ \ \  \exists \ \  C >0$ such that
$$
 \tilde{w} ^\pm_\la (y)  \leq   C e^{ \frac{\sqrt{1-\de}}{R_2}} |y| $$
for $y\in \tilde{D}^+_\la =\left\{ y \in \R^N :   P^\pm_\la  + \frac{y}{\sqrt{\la}}   \in D \right\}$

Then, as in \cite{NW} we get the assertion ii) -  To conclude we prove that the limit points $P^+ $ and  $P^-$ belong to the
 inner  boundary  of $D =\{z \in \pa D, |z| = R_1\}$ i.e claim  iii).

To do this,  we simplify also the proof of \cite{RS} for the least energy  positive solution  exploiting the fact that  our   domain is an  annulus centred at the
origin. Let us   observe that the energy functional  (4.2) on the  solution $\tilde{v}_\la$ can be written as
$$J_\la (\tilde{v}_\la) =J_\la (\tilde{v}_\la^+) +J_\la  (\tilde{v}_\la ) $$
Then  by rescaling  $\tilde{v}_\la^\pm$ about  $P^\pm_\la $ in the usual way  we obtain, as in \cite{Nou-W}  and arguing as in \cite{NW}, that
$$
J_\la (u^\pm_\la ) =  I_{d^\pm} (w_{d^\pm} ) + o (1)  \mbox{ as } \la \rightarrow  \infty$$
where $w_{d^+}$ and $w_{d^-}$ are the positive solutions  of
\be \left\{ \begin{array}{llll}
-\De w + \frac{l}{2d} w -  \frac{1}{2d} w^p =0   \mbox{ in } \R^N \\[2mm]
 w(x) \rightarrow 0 \mbox{ as } |x| \rightarrow  \infty \end{array} \right. \ee

with  $d = d^+ = |P^+|$ or $d = d^- =  |P^-| $ and  $P^+_\la  \rightarrow   P^+, P^-_\la \rightarrow P^-$  respectively, and
$$
I_{d^\pm} \left( w_{d^\pm} \right) =\frac{1}{2} \int\limits_{\R^N}  |\na w_{d^\pm}| d x + \frac{1}{2d ^\pm}
\int\limits_{\R^N} \frac{1}{2} |w_{d^\pm} |^2 dx - \frac{ 1}{2d}
\int\limits_{\R^N}   \frac{1}{p+1} |w_{d^\pm} |^{p+1} dx $$
As  observed in \cite{RS} we have
\be
I_{d^\pm} (w_{d^\pm}) =  \sqrt{2d^\pm} I (z) \ee
where $z$ is the solution  of the equation
$$- \De z + z -z^p =0   \mbox{ on } \R^N$$

 since   $w_{d^\pm} (|x|) = z \left(\frac{|x|}{\sqrt{2d^\pm}} \right). $

\par From (4.8) it is easy to understand  that  in order to reduce the energy  the points    $P^\pm_\la$ should
converge to points $P^\pm$ in the annulus which have the smallest distance from the origin. These are in fact the points on the
 inner boundary.

To prove it  rigorously assume that one of the two   points  $\{P^\pm_\la\}$, say $P^+_\la$ converge   to a point $P^+ \in \overline{D}$ with   $|P^+| =R_1 +\alpha $ for some $\al >0.$ Then we could consider  the ball $B (Q, \al/3)$ with
   center in $Q= (0, \ldots,  0, R_1 +\frac{\al}{3}) $ and define the function
$$h^+_\la (x) =\varphi (x) w_{d_\al} ((x -Q) \sqrt{\la} ) \mbox{ for } x \in D.$$
where  $w_{d_\al}$ is the solution   of (4.6) with $d= d_\al = R_1 +
\frac{\al}{3}$  and $\varphi$ is a  suitable cut off function such that $h^+_\la  \in C^2_0
(B (Q, \al/ 3)).$
Hence  $h_\la^+ (x) $ is  just the solution of (4.6) (for  $(d_\al = R_1 +\frac{\al}{3})$ suitably translated  rescaled and  cut to be defined  in D. Then
\be
 J_\la (h_\la^+ )\rightarrow I_{d_\al} (w_{d_\al})=\sqrt{2d}_\al I (z) < \sqrt{2d^+} I (z)  \ee
 as  $\la   \rightarrow \infty$, by the  choice  of $d_\alpha.$
Then we consider  the function
$$h^-_\la (x) = \psi _\la  (x) \tilde{v}_\la^ -  (x) $$
where  $\psi_\la$ is a suitable cut off function such that $h^-_\la$ and $h^+_\la$ have disjoint supports. Note that
 this can be  always done since
$|P^+_\la -P^-_\la  |  > 2 R_1$ as pointed out before .

Finally, we can have that the function
$$h_\la (x) =h_\la ^+ (x) -h^-_\la (x)$$
belong to the  nodal Nehari set  (4.3). Then we  get
$$ J_\la (h_\la) \rightarrow I_{d_\al} (w_{d_\al}) + I_{d^- }(w_{d^-}) <  I_{d^+} + ( w_{d^+}) + I_{d^-} (w_{d^-} ) \mbox{ as }
\la  \rightarrow  \infty$$

by comparison with (4.8) and (4.9). For $\la$  large this contradicts  the fact that $\tilde{v}_\la$ is the
least energy nodal solution  as stated in  Proposition 4.1. Hence (iii) is proved.
\qed
\section{ Proofs of Theorem 1.2 and  Theorem 1.3}
 We start with the proof of Theorem  1.2 which,  at this stage, is a direct  consequence of the results of the
previous  sections.

 \noi {\bf Proof of Theorem 1.2} Let $1 <p < \frac{(m+1) +2}{ (m+1)-2}$ and  consider  problem (4.1) with such
 exponents  $p$ in the  annulus $D \subset \R^{m+1}, (i.e \  N =m+1,$
 in  Section 4) with  radii    $R_1 =\frac{a^2}{2},  R_2 =
\frac{b^2}{  2}$. By Proposition 4.1 there exist two families of solutions :$\{ v_\la\}$
 and  $\{ \tilde{v}_\la\}$ which are foliated Schwarz symmetric, hence, in   particular, they are axially symmetric and so belong
to the space $Y.$

\par Thus Theorem 1.1 applies and   we get families of solutions  $\{u_\la\}$ and $\{\tilde{u}_\la\}$ in $X$ for
problem  (1.3). Finally, by Theorem 4.2 and Theorem 4.3,  we have  that ${v_\la}$  concentrates in  a point $P_\la$   while
$\tilde{v}_\la$ concentrate  in two points   $P^+_\la $ and $P^-_\la.$ All   these points belong  to the symmetry axis, which
 is the  $z_{m+1} $ axis and converge to points $P, P^+$ and $P^-$  lying  on the inner boundary  of $D.$ Therefore by the  transportation   map  (3.3) and Remark  3.1
we get that $u_\la$ and $\tilde{u}_\la$ have the  claimed  concentration properties. \qed.

\par Next we prove that the Morse indices $m (u_\la)$ of $u_\la$ and   $ m (\tilde{u_\la})$   of $ \tilde{u}_\la$  tend  to
infinity  as $\la  \rightarrow  +\infty.$
To this  aim let us set
\setcounter{equation}{0}
\be
 L_{u_\la} = - \De + \la I -  pu_\la ^{p-1} I  \ee
and
\be
L_{\tilde{u}_\la} = -\De + \la I - p  |\tilde{u}_\la |^{p-1} I \ee

the linearised operators  at $u_\la$ and $\tilde{u}_\la $ and define the associated  quadratic forms:
\be
Q_{u_\la} (\psi) = \int\limits_A |\na \psi|^2 dx +\la  \int\limits_A  |\psi|^2 dx -p  \int\limits_A u^{P-1}_\la
  \psi ^2 dx \ee
for  $\psi \in  H_0^1 (A)$ and  $Q_{\tilde{u}_\la} (\psi) $ defined analogously.

Let us   denote by $\mu_j =\mu_j (\la)$ (respectively, $\tilde{\mu}_j = \tilde{\mu}_j (\la))$ the
 eigenvalues  of $L_{u_\la}$ (resp.  $L_{\tilde{u}_\la} $) in  $H^1_0 (A), j \in \N.$
We    have
\begin{lemma}  The eigenvalues $\mu_1, \tilde{\mu}_1, \tilde{\mu}_2$ tend  to $-\infty$ as $\la
\rightarrow  \infty.$ \end{lemma}

\begin{proof}
Let  us show it for $\mu_1.$ We evaluate  the quadratic form (5.3) on $u_\la$ itself. By the equation (1.3). We have

$$ Q_{u_\la} (u_\la) = (1-p) \int\limits_A (|\nabla u_\la)|^2 + \la  |u_\la|^2 ) dx $$
Hence

$$ \mu_1 \leq  \frac{Q_\la (u_\la )}{\int\limits_A |u_\la|^2 dx } =  (1-p)  \left[
\frac{\int\limits_A |\nabla u_\la|^2 dx }{\int\limits_A |u_\la |^2 dx } +\la \right]  \leq
 (1-p) \la \rightarrow -\infty  \mbox{ as } \la \rightarrow +\infty.   $$
The same holds  for $\tilde{\mu}_i, i =1, 2, $ using  $\tilde{u}^+_\la$ and $\tilde{u}_\la^-$ as test
functions  to evaluate the quadratic form .
\end{proof}

 To show the asymptotic behavior of the  Morse index of our solutions we
 construct  a sequence  $\{\Phi_k\}$  of
$L^2 $-orthogonal functions,   on which the  quadratic form (5.3) is negative for $\la$ large.

\par We need some  preliminary notations and remarks. As in Section 3 a point $x \in \R^{2m} = \R^m \times \R^m$ is
represented  by $x = (y_1, y_2), y_i \in \R^m,  i =1, 2.$ Then  $y_1 = (\rho_1, \sigma _1), \ \ \  y_2 = (\rho_2, \sigma_2)$ with
  $\rho_i  =|y_i|, \  \ \  \sigma_i \in S^{m-1} \subset \R^m  (y_i), \ \ \  i =1, 2 $ and  $\rho_1 =  r \cos \theta,  \
\rho_2 = r \sin \theta,     \  \
 r =|x|, \ \ \ta  \in [ 0, \frac{\pi}{2} ]$.  Thus we can represent $x \in \R^{2m} $ by
$$ x = (r, \sigma_1,  \sigma_2, \ta).$$

%\end{document}
Then  the Laplace operator in $\R^{2m}$ can be  expanded as:
\be
\De_{R^{2m}}  u = u_{rr} +\frac{(2m -1)}{r} u_r + \frac{(m-1)}{r^2} u_\ta
\left[ \frac{2 \cos 2 \ta}{\sin  2 \ta} \right] +  \frac{u_{\ta \ta}}{r^2} +
\frac{1}{r^2 \cos^2 \ta }  \De^{\sigma_1}_{S^{m-1}}  u +  \frac{1}{r^2 \sin^2  \ta }
\De^{\sigma_2}_{S^{m-1}} u \ee
where  $\De ^{\sigma_i}_{S^{m-1}}  , i =1, 2$ is the Laplace-Beltrami   operator on $S^{m-1}$ in the
 $\sigma_i$ -variable.  Since the   solutions $u_\la$ and $\tilde{u}_\la$ are radially symmetric in $y_1$ and  $y_2.$, ie.
belong   to the  space $X$ (see 3.2) the linearized operators  are invariant under the same symmetry. Denoting
by $g_1 = g_1 (\la)$ (resp. $\tilde{g}_1 =\tilde{g}_1 (\la))$ the first $ L^2$- normalized eigenfunction of $L_{u_\la}$
(resp  $L_{\tilde{u}_\la})$ in $H^1_0 (A)$  easily have

\begin{lemma}
The eigenfunctions $g_1$ and $\tilde{g}_1$ belong to $H_0^1 (A) \cap X$ i.e  depend  only on $(r, \theta). $
\end{lemma}

\begin{proof}

Since  $u_\la $ and $\tilde{u}_\la$ belong to $X$ then $L_{u_\la}$ and $L_{\tilde{u}_\la}$ are invariant by the same
symmetry. Therefore  the  symmetry of $g_1$ and $\tilde{g}_1$ derives by the uniqueness of the first  eigenfunction
(up to normalization). Indeed  the first eigenfunction of $L_{u_\la}$ and $L_{\tilde{u}_\la}$ in  $H_0^1 (A) $ and
$H_0^1 (A) \cap X $ must be the same.
 \end{proof}
Let us observe that, since $g_1$ depends only on $(r, \ta)$ it satisfies the problem
\be \left\{ \begin{array}{llll}
-(g_1)_{rr}  - \frac{2 (m-1)}{r} (g_1)_r - \frac{m-1}{r^2} (g_1)_\ta  \ \ \ \frac{2 \cos  2 \ta}{\sin 2 \ta}
- \frac{(g_1)_{\ta \ta}}{r^2} +\la  g_1  - p u_\la^{ p-1} g_1 =\mu_1 g_1\\[2mm]
g_1 =0  \quad \mbox{ on } \pa A \end{array}\right. \ee
The  analogous  statement holds for $\tilde{g}_1. $ Then  let  $\psi_k$ be  the  $k$-th  eigenfunction of
$-\De_{s^{m-1}}$  corresponding  to the eigenvalue $\nu_k  =k
 (k+m-2), m \geq 2, k \geq 1. $ We have
\begin{lemma}
Define for any  $k\geq 1 \quad$
\be
   \Phi^k  =g_1 (r, \ta) \left[ \cos^2\ta \psi_k (\sigma_1) +\sin^2 \ta \psi_k
(\sigma_2)\right]  \ee

Then, for  any $k\geq 1, $
$$ Q_{u_\la} (\Phi^k) <0 ,  \ \ \  \mbox{ for } \la  \ \ \   \mbox{sufficiently large.}$$

\end{lemma}

\begin{proof}

 Note  that, by Lemma 5.2,  only $g_1$ depends  on $r$ while only           $\psi_k$ depends  on $\sigma_1$ or $\sigma_2$. Hence, the relevant terms in evaluating the quadratic form come only  from the $\theta$ derivatives. Then   by (5.4), (5.5), since  $\psi_k (\sigma_1), \psi_k (\sigma _2)$ are  eigenfunctions  of
$-\De_{ s^{m-1}}$ corresponding  to the same   eigenvalue $\nu_k$ we obtain
$$L_{ u_{\la}}  \Phi^k = \mu_1  \Phi^k  + \frac{\nu_k}{r^2} g_1  \psi_k (\sigma_1) +
\frac{\nu_k}{r^2} g_1 \psi_k (\sigma_2) - \frac{(m-1)}{r^2} g_1 2 \cos 2 \theta
[\psi_k (\sigma_2) - \psi_k (\sigma_1) ]$$
$$+ \frac{2}{r^2} [(g_1)_\theta \sin  2 \theta + g_1 \cos 2 \theta ]
[\psi_k (\sigma_2) -\psi_k (\sigma_1)] $$
Using that $g_1,\psi_k$ can be taken as $L^2$-normalized  and that $\psi_k$ has mean value
 zero  on $S^{m-1}$, i.e
$$
\int\limits_A g_1^2 =1 \quad \int\limits_{S^{m-1}} |\psi_k|^2  =1 \mbox{ and } \int\limits_{S^{m-1}} \psi_k =0$$
and multiplying by $\Phi^k$ and integrating  we get:
$$Q_{u_\la} (\Phi^k) = \langle L_{u_\la} \Phi^k, \Phi^k \rangle =\mu_1 \int\limits_A  |\Phi^k|^2 dx +
\nu_k \int\limits_A \frac{1}{r^2} g_1^2 [\psi^2_k (\sigma_1) \cos^2 \ta   +\psi^2_k (\sigma_2)
\sin^2 \ta ] dx + $$
$$ -2 (m-1) \int\limits_A \frac{1}{r^2}  g_1^2 \cos 2 \ta [\psi_k (\sigma_2) -\psi_k (\sigma_1)]
[\cos^2 \ta  \ \psi_k (\sigma_1) + \sin^2 \ta   \ \psi_k (\sigma_2)] dx $$
  $$+  \int\limits_A \frac{2}{r^2}  (g_1)_\ta  \cdot   g_1 \sin 2 \ta
 [\psi_k (\sigma_2) -\psi_k (\sigma_1)]
    [\cos^2 \ta \   \psi_k  (\sigma_1) +\sin^2  \ta \   \psi_k (\sigma_2)] dx $$
$$+\int\limits_A \frac{2}{r^2} (g_1)^2 \cos 2 \ta  [(\psi_k (\sigma_2) -\psi_k (\sigma_1)] [\cos^2 \ta   \ \psi_k (\sigma_1) +
\sin^2 \ta \   \psi_k (\sigma_2)]dx $$
For the first term we have, by the previous  remarks, \\
$$\int\limits_{A} |\Phi^k|^2 =\int\limits_A g_1^2 [\cos^4 \ta +\sin^4 \ta ]\geq  \de > 0 $$
where   $\de =\min\limits_{[0, \pi/2]}(sin^4 \ta +\cos^4 \ta) >0$
Then, taking  into account that $2(g_1)_\ta\cdot g_1 =(g_1^2)_\ta$, we finally  get
$$
Q_{u_\la} (\Phi^k) \leq \mu_1 \de + C_k \int\limits_A |g_1|^2  =\mu_1 \de + C_k $$
for some constant $C_k$ independent of $\la.$ Since $\mu_1 =\mu_1 (\la) \rightarrow  -\infty$ as  $\la
 \rightarrow \infty$ we get the
assertion .
\end{proof}

Of course the statement  of Lemma 5.3 holds also if we substitute $g_1$ with $\tilde{g}_1.$ Thus

\noi {\bf Proof of Theorem 1.3.}
We consider the sequences $\Phi^k$ defined by (5.8) and $\tilde{\Phi}^k =\tilde{g}_1 (r, \ta) [\cos^2 \ta  \ \psi_k
(\sigma_1) + \sin^2 \ta \   \psi_k (\sigma_2)], k \geq 1$ and observe that
$$
\int\limits_A \Phi^k\Phi^j  dx =\int\limits_A \tilde{\Phi}^k. \tilde{\Phi}^j =0  \mbox{ for } j \neq k$$

By Lemma 5.3, for any $k \geq 1 $ there exists $\la (k) $ such that $Q_{u_\la} (\Phi^k ) < 0 $ for  $\la  >\la (k). $ Thus  for $\la \rightarrow  -\infty,$ the Morse index $m (u_\la)$ tends to infinity.The same applies to $m (\tilde{u}_\la).$

\qed

\end{document}